\documentclass[10pt]{amsart}
\title[K-theory for $\widetilde A_2$ groups]{Harmonic cochains and K-theory for $\widetilde A_2$ groups}
\author{Guyan Robertson}
\address{School of Mathematics and Statistics, University of Newcastle, NE1 7RU, U.K.}
\email{a.g.robertson@ncl.ac.uk}

\usepackage{amsmath}
\usepackage{amscd}
\usepackage{amssymb}

\input{prepictex}
\input{pictex}
\input{postpictex}
\chardef\bslash=`\\

\makeatletter
\def\verbatim{\interlinepenalty\@M \@verbatim
  \leftskip\@totalleftmargin\advance\leftskip2pc
  \frenchspacing\@vobeyspaces \@xverbatim}
\makeatother
\newtheorem{theorem}{Theorem}[section]

\newtheorem{lemma}[theorem]{Lemma}
\newtheorem{proposition}[theorem]{Proposition}

\theoremstyle{definition}

\newtheorem{remark}[theorem]{Remark}

\newcommand{\cl}[1]{{\mathcal{#1}}}
\newcommand{\bb}[1]{{\mathbb{#1}}}
\newcommand{\fk}[1]{{\mathfrak{#1}}}
\newcommand{\ang}[1]{{\langle{#1}\rangle}}
\newcommand{\ovl}{\overline}

\newcounter{picture}

\newcommand{\e}{{\varepsilon}}

\newcommand{\PGL}{{\text{\rm{PGL}}}}

\begin{document}

\begin{abstract}
If $\Gamma$ is a torsion free $\widetilde A_2$ group
acting on an $\widetilde A_2$ building $\Delta$,
and $\fk A_{\Gamma}$ is the associated boundary $C^*$-algebra,
it is proved that
$K_0(\fk A_\Gamma)\otimes \bb R \cong \bb R^{2\beta_2}$,
  where $\beta_2=\dim_\bb R H^2(\Gamma, \bb R)$.
\end{abstract}

\maketitle

\section{Introduction}
Let $\Gamma$ be an $\widetilde A_2$ group acting on an $\widetilde A_2$ building $\Delta$ of order $q$. The Furstenberg boundary $\Omega$ of $\Delta$ is the set of chambers of the spherical building at infinity, endowed with a natural compact totally disconnected topology.
The topological action of $\Gamma$ on $\Omega$ is encoded in the full crossed product $C^*$-algebra $\fk A_{\Gamma}=C(\Omega)\rtimes\Gamma$, which is studied in \cite{rs1,rs2,rs3}.
This full crossed product is isomorphic to the reduced crossed product, since the action of $\Gamma$
on $\Omega$ is amenable \cite[Section 4.2]{rs1}.
As the notation suggests, $\fk A_\Gamma$ depends only on $\Gamma$ \cite{rs3}.
Motivated by rigidity theorems of Mostow, Margulis and others, whose proofs rely on the study of boundary actions, it is of interest to determine the extent to which
the boundary $C^*$-algebra $\fk A_\Gamma$ determines the group $\Gamma$.

In \cite{rs3}, T. Steger and the author computed the K-theory of $\fk A_\Gamma$ for  many $\widetilde A_2$ groups with $q\le 13$. The computations were done for all the $\widetilde A_2$ groups in the cases $q=2,3$ and for several representative groups for each of the other values of $q\le 13$.
If $q=2$ there are precisely eight $\widetilde A_2$ groups $\Gamma$, all of which embed as lattices in $\PGL (3,\bb K)$, where $\bb K= \bb F_2((X))$ or $\bb K=\bb Q_2$. If $q=3$ there are 89 possible $\widetilde A_2$ groups, of which 65 are ``exotic'' in the sense that they do not embed naturally in linear groups.  Exotic $\widetilde A_n$ groups only exist if $n=2$, since all locally finite Euclidean buildings of dimension $\geq 3$ are associated to linear algebraic groups. This justifies, to some extent, the focus on $\widetilde A_2$ groups.

  For each $\widetilde A_2$ group $\Gamma$, the $C^*$-algebra $\fk A_{\Gamma}$ has the structure of a rank 2 Cuntz-Krieger algebra \cite[Theorem 7.7]{rs2}. These algebras are classified up to isomorphism by their $K$-groups \cite[Remark 6.5]{rs2} and it was proved in \cite[Theorem 2.1]{rs3} that
  \begin{equation}\label{Kremark}
  K_0(\fk A_{\Gamma})=K_1(\fk A_{\Gamma})=\bb Z^{2r}\oplus T,
  \end{equation}
  where $r\ge 0$ and $T$ is a finite group. The computations in \cite{rs3}
led to some striking observations. For example, the three torsion-free $\widetilde A_2$ subgroups of $\PGL_3(\bb Q_2)$ are distinguished from each other by $K_0(\fk A_{\Gamma})$.
    There was also evidence for the conjecture that, for any torsion free
  $\widetilde A_2$ group $\Gamma$,
  the integer $r$ in the equation (\ref{Kremark}) is equal to the second Betti number of $\Gamma$.
  The purpose of this article is to prove that this is indeed the case.

\begin{theorem}\label{main}
If $\Gamma$ is a torsion free $\widetilde A_2$ group
acting on an $\widetilde A_2$ building $\Delta$ of order $q$, then
\begin{equation}
K_0(\fk A_\Gamma)\otimes \bb R \cong \bb R^{2\beta_2},
\end{equation}
  where $\beta_2=\dim_\bb R H^2(\Gamma, \bb R)=\frac{1}{3}(q-2)(q^2+q+1)$.
\end{theorem}

The article \cite{rs3} identified the integer $r$ in (\ref{Kremark}) with
the rank of a certain finitely generated abelian group $C(\Gamma)$. Two new ideas lead to the proof of Theorem \ref{main}.
The local structure of the building $\Delta$, together
with the fact that $\Gamma$ has Kazhdan's property (T), is used to show that
$C(\Gamma)\otimes\bb R$ is isomorphic to the space of $\Gamma$-invariant $\bb R$-valued cochains on $\Delta$, in the sense of \cite{ads,eck}. Then, according to an isomorphism of Garland \cite{gd}, this space is isomorphic to $H^2(\Gamma,\bb R)$.

\begin{remark} An $\widetilde A_2$ group is a natural analogue of a free group, which acts freely and transitively
on the vertex set of a tree (which is a building of type $\widetilde A_1$).
If the tree is homogeneous of degree $q+1$, with $q\ge 2$, then
$\Gamma$ is a free group on $\frac{1}{2}(q+1)$ generators and one can again form the full crossed product $C^*$-algebra $\fk A_{\Gamma}=C(\Omega)\rtimes\Gamma$, where $\Omega$ is the space of ends of the tree. The analogue of Theorem
\ref{main} states that
$K_0(\fk A_\Gamma)\otimes \bb R \cong \bb R^{\beta_1}$,
  where $\beta_1=\dim_\bb R H^1(\Gamma, \bb R)=\frac{1}{2}(q+1)$
\cite[Theorem 1]{rtree}.
\end{remark}

\begin{remark}
Another simple $C^*$-algebra associated with  the $\widetilde A_2$ group $\Gamma$ is the reduced group
$C^*$-algebra $C_r^*(\Gamma)$.
It is shown in \cite[Theorem 6.1]{ro2005} that
$K_0(C_r^*(\Gamma)) = \bb Z^{\chi(\Gamma)}$.
This is a consequence of the fact that $\widetilde A_2$ groups
belong to the class of groups for which the Baum-Connes conjecture is known to be true.
\end{remark}

\begin{remark}
This paper is a sequel to the articles \cite{rs2}, \cite{rs3}. 
The key results used are \cite[Theorem 7.7]{rs2}, which shows that
$\fk A_{\Gamma}$ is isomorphic to a rank 2 Cuntz-Krieger algebra, and \cite[Theorem 2.1]{rs3} which shows that the $K$-theory of this algebra is given by equation (\ref{Kremark}).

What happens in the case of a torsion free $\widetilde A_n$-group $\Gamma$ ($n\ge 3$)?
There seems to be no fundamental obstruction to generalising \cite[Theorem 7.7]{rs2}, to identify the boundary crossed product algebra with a higher rank Cuntz-Krieger algebra, in the sense of
\cite{rs2}. The arguments of the present paper should also generalise, but additional conditions which are vacuous in the rank 2 case would need to be verified \cite[Theorem 2.3 (C),(D)]{ads}. 
However it would be more difficult to generalise \cite[Theorem 2.1]{rs3}. This is because the proof of that result uses 
a Kasparov spectral sequence \cite[Proposition 4.1]{rs3} whose limit is
clear only in the rank 2 case.
\end{remark}

\section{$\widetilde A_2$ groups}\label{A2tildegroups}

Consider a locally finite building $\Delta$ of type $\widetilde A_2$.  Each vertex $v$ of $\Delta$ has a type $\tau (v) \in \bb Z/3\bb Z$,
and each chamber (maximal simplex) of $\Delta$ has exactly one vertex of each type.
Each edge $e$ is directed, with initial vertex of type $i$ and final vertex of type $i+1$.
An automorphism $\alpha$ of $\Delta$ is \textit{type rotating} if there exists $i \in
\bb Z/3\bb Z$ such that $\tau(\alpha(v)) = \tau(v)+i$ for all vertices $v$
of $\Delta$.

Suppose that $\Gamma$ is a group of type rotating automorphisms of $\Delta$, which acts freely and transitively on the vertex set of $\Delta$. Such a group is called an $\widetilde A_2$ group. The theory of $\widetilde A_2$ groups has been developed in \cite{cmsz} and some, but not all, $\widetilde A_2$ groups embed as lattice subgroups of $\PGL_3(\bb K)$.
Any $\widetilde A_2$ group can be constructed as follows \cite [I, Section 3]{cmsz}.
Let $(P,L)$ be a finite projective plane of order $q$. There are $q^2 + q + 1$
points
 (elements of $P$) and $q^2+q+1$
lines (elements of $L$).
Let $\lambda : P \rightarrow L$ be a bijection and write $\lambda(\xi)=\ovl {\xi}$.
A {\it triangle presentation}  compatible with $\lambda$ is a set ${\mathcal T}$ of ordered triples $({\xi_i}, {\xi_j}, {\xi_k})$
where ${\xi_i}, {\xi_j}, {\xi_k} \in P$, with the following properties.
\begin{itemize}
\item[(i)]  Given ${\xi_i}, {\xi_j} \in P$, then $({\xi_i}, {\xi_j}, {\xi_k}) \in {\mathcal T}$ for some
${\xi_k} \in P$ if and only if ${\xi_j}$ and $\ovl {\xi_i}$ are incident, i.e. ${\xi_j} \in
\ovl {\xi_i}$.

\item[(ii)] $({\xi_i}, {\xi_j}, {\xi_k}) \in {\mathcal T} \Rightarrow ({\xi_j}, {\xi_k}, {\xi_i}) \in {\mathcal T}$.

\item[(iii)]  Given ${\xi_i}, {\xi_j} \in P$, then $({\xi_i}, {\xi_j}, {\xi_k}) \in {\mathcal T}$ for at
most one ${\xi_k} \in P$.
\end{itemize}
In \cite{cmsz} there is exhibited a complete list of triangle
presentations for $q = 2$
and $q = 3$.
Given a triangle presentation $\cl T$, one can form the group
\begin{equation}\label{rel}
\Gamma=\Gamma_{\cl T}
 = \big\langle P \ |\  {\xi_i} {\xi_j} {\xi_k} = 1  \hbox { for } ({\xi_i},
{\xi_j}, {\xi_k}) \in {\mathcal T}
\big \rangle .
\end{equation}
\noindent The Cayley graph of $\Gamma$ with respect to the
generating set $P$ is the $1$-skeleton of a building $\Delta$ of type $\widetilde A_2$.
Vertices are elements of $\Gamma$ and a directed edge of the form $(\gamma,\gamma\xi)$ with $\gamma\in\Gamma$ is labeled by the generator $\xi\in P$.

The link of a vertex $\gamma$ of $\Delta$ is the incidence graph of the projective plane $(P,L)$, where
the lines in $L$ correspond to the inverses in $\Gamma$ of the
generators in $P$. In other words,
$\ovl \xi = \xi^{-1}$ for $\xi \in P$.

\refstepcounter{picture}
\begin{figure}[htbp]
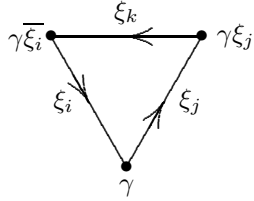
\label{A4}
\hfil
\centerline{
\beginpicture
\setcoordinatesystem units <1cm, 1.732cm>
\put {$\bullet$} at 0 0
\multiput {$\bullet$} at -1 1 *1 2 0 /
\put {$\gamma$} [t] at 0 -0.1
\put {$\gamma\ovl {\xi_i}$} [r] at -1.1 1
\put {$\gamma {\xi_j}$} [l] at 1.2 1
\arrow <10pt> [.2, .67] from  0.2 1 to 0 1
\arrow <10pt> [.2, .67] from  -0.7 0.7 to -0.5 0.5
\arrow <10pt> [.2, .67] from  0.3 0.3 to 0.5 0.5
\put {${\xi_i}$} [r ] at -0.7 0.5
\put {${\xi_j}$} [ l] at 0.7 0.5
\put {${\xi_k}$} [ b] at 0 1.1
\putrule from -1 1 to 1 1
\setlinear \plot -1 1 0 0 1 1 /
\endpicture
}
\hfil
\caption{A chamber based at a vertex $\gamma$}
\end{figure}

For the rest of this article, $\Gamma$ is assumed to be torsion free. 
Therefore $\Gamma$ acts freely on $\Delta$ and $X=\Gamma\backslash\Delta$ is a $2$-dimensional cell complex with universal covering $\Delta$. Let $X^k$ denote the set of oriented $k$-cells of $X$ for $k=0,1,2$.
Thus $X^1$ may be identified with $P$ and
$X^2$ may be identified with the set of orbits of elements of $\cl T$ under cyclic permutations.

Let $\widehat\Delta^2$ be the directed version of $\Delta^2$ in which each 2-simplex has a
specified base vertex, so that $\bb Z/3\bb Z$ acts naturally on $\widehat \Delta^2$.
Let $\widehat X^2:=\widehat\Delta^2/\Gamma$, the set of directed 2-cells of $X$.
Then $\widehat X^2$ may be identified with $\cl T$.
From now on $a= \ang{a_0,a_1,a_2}$ will denote an element of $\cl T$, regarded as a directed 2-cell. Figure \ref{directed chamber} illustrates the three directed 2-cells
associated with an oriented 2-cell of $X$. In the diagram, the 2-cells are thought of as being directed upwards and the symbol $\bullet$ is placed
opposite the ``top'' edge to indicate that direction.

\refstepcounter{picture}\label{directed chamber}
\begin{figure}[htbp]
\hfil
\centerline{
\beginpicture
\setcoordinatesystem units <1cm, 1.732cm>
\setplotarea x from -2 to 2, y from 0 to 1 
\arrow <8pt> [.2, .67] from  0.2 1 to 0 1
\arrow <8pt> [.2, .67] from  -0.7 0.7 to -0.5 0.5
\arrow <8pt> [.2, .67] from  0.3 0.3 to 0.5 0.5
\put {$\bullet$} at 0 0.3
\put {$a_1$} [r ] at -0.7 0.5
\put {$a_2$} [ l] at 0.7 0.5
\put {$a_0$} [ b] at 0 1.1
\putrule from -1 1 to 1 1
\setlinear \plot -1 1 0 0 1 1 /
\endpicture
\beginpicture
\setcoordinatesystem units <1cm, 1.732cm>
\setplotarea x from -2 to 2, y from 0 to 1 
\arrow <8pt> [.2, .67] from  0.2 1 to 0 1
\arrow <8pt> [.2, .67] from  -0.7 0.7 to -0.5 0.5
\arrow <8pt> [.2, .67] from  0.3 0.3 to 0.5 0.5
\put {$\bullet$} at 0 0.3
\put {$a_2$} [r ] at -0.7 0.5
\put {$a_0$} [ l] at 0.7 0.5
\put {$a_1$} [ b] at 0 1.1
\putrule from -1 1 to 1 1
\setlinear \plot -1 1 0 0 1 1 /
\endpicture
\beginpicture
\setcoordinatesystem units <1cm, 1.732cm>
\setplotarea x from -2 to 2, y from 0 to 1 
\arrow <8pt> [.2, .67] from  0.2 1 to 0 1
\arrow <8pt> [.2, .67] from  -0.7 0.7 to -0.5 0.5
\arrow <8pt> [.2, .67] from  0.3 0.3 to 0.5 0.5
\put {$\bullet$} at 0 0.3
\put {$a_0$} [r ] at -0.7 0.5
\put {$a_1$} [ l] at 0.7 0.5
\put {$a_2$} [ b] at 0 1.1
\putrule from -1 1 to 1 1
\setlinear \plot -1 1 0 0 1 1 /
\endpicture
}
\hfil
\caption{The directed 2-cells $\ang{a_0,a_1,a_2},\ang{a_1,a_2,a_0},\ang{a_2,a_0,a_1}$.}
\end{figure}

\section{K-theory}

Transition matrices $M=(m_{ab})_{a,b\in \widehat X^2}$ and
$N=(n_{ab})_{a,b\in \widehat X^2}$ are defined as follows. If $a,b\in \widehat X^2$ then $m_{ab}=1$ if and only if there are labeled triangles representing $a,b$ in the building $\Delta$ which lie as shown on the right of Figure \ref{transition}.
If no such diagram is possible then $m_{ab}=0$.

In terms of the projective plane $(P,L)$, the matrix $M$ is defined by
\begin{equation*}
m_{ab}=1\ \Longleftrightarrow \ b_2 \notin \ovl a_2,\  \ovl b_1=a_0\vee b_2.
\end{equation*}
It follows that each row or column of $M$ has precisely $q^2$ nonzero entries.

\refstepcounter{picture}
\begin{figure}[htbp]
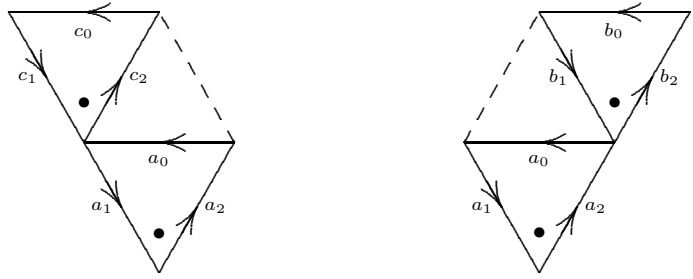
\label{transition}
\hfil
\centerline{
\beginpicture  
\setcoordinatesystem units <1cm, 1.732cm>
\setplotarea  x from -2 to 2,  y from -1 to 1
\put {$\bullet$} at 0 -0.7
\put {$\bullet$} at -1 0.3
\put {$_{a_1}$} [r ] at -0.6 -0.5
\put {$_{a_2}$} [ l] at 0.6 -0.5
\put {$_{c_1}$} [r ] at -1.6  0.5
\put {$_{c_2}$} [ l] at -0.4  0.5
\put {$_{a_0}$} [t] at  0.0   -0.1
\put {$_{c_0}$} [b] at  -1   0.8
\multiput {\beginpicture
\setcoordinatesystem units <1cm, 1.732cm>
\arrow <10pt> [.2, .67] from  0.2 1 to 0 1
\arrow <10pt> [.2, .67] from  -0.7 0.7 to -0.5 0.5
\arrow <10pt> [.2, .67] from  0.3 0.3 to 0.5 0.5
\putrule from -1 1 to 1 1
\setlinear \plot -1 1 0 0 1 1 /
\endpicture}  at   0 0  -1 1  /
\setdashes \plot 1 0  0 1 /
\endpicture
\qquad\quad   
\beginpicture  
\setcoordinatesystem units <1cm, 1.732cm>
\setplotarea  x from -2 to 2,  y from -1 to 1
\put {$\bullet$} at 0 -0.7
\put {$\bullet$} at 1 0.3
\put {$_{a_1}$} [r ] at -0.6 -0.5
\put {$_{a_2}$} [ l] at 0.6 -0.5
\put {$_{b_1}$} [r ] at 0.4  0.5
\put {$_{b_2}$} [ l] at 1.6  0.5
\put {$_{a_0}$} [t] at  0.0   -0.1
\put {$_{b_0}$} [b] at  1.0   0.8
\multiput {\beginpicture
\setcoordinatesystem units <1cm, 1.732cm>
\arrow <10pt> [.2, .67] from  0.2 1 to 0 1
\arrow <10pt> [.2, .67] from  -0.7 0.7 to -0.5 0.5
\arrow <10pt> [.2, .67] from  0.3 0.3 to 0.5 0.5
\putrule from -1 1 to 1 1
\setlinear \plot -1 1 0 0 1 1 /
\endpicture}  at   0 0  1 1  /
\setdashes \plot -1 0  0 1 /
\endpicture
}
\hfil
\caption{The relations $n_{ac}=1$ and $m_{ab}=1$}
\end{figure}

\noindent Similarly, the matrix $N$ is defined by
\begin{equation*}
n_{ac}=1\ \Longleftrightarrow \ a_1 \notin \ovl c_1,\ c_2=\ovl a_0\wedge \ovl c_1.
\end{equation*}
as illustrated on the left of Figure \ref{transition}.

Let $r$ be the rank, and $T$ the torsion part, of the abelian group $C(\Gamma)$ with generating set $\widehat X^2$ and relations

\begin{equation}\label{rels}
a=\sum_{b\in \widehat X^2}m_{ab}b =\sum_{b\in \widehat X^2}n_{ab}b,\quad a\in \widehat X^2.
\end{equation}
Thus
$C(\Gamma)\cong\bb Z^r\oplus T$.
The following result was proved in
\cite[Theorem 2.1]{rs3}.

\begin{theorem}\label{Ki}
Let $\Gamma$ be an $\widetilde A_2$ group, and let $r$ be the rank, and $T$ the torsion part of
$C(\Gamma)$.
Then
\begin{equation}\label{K}
K_0(\fk A_\Gamma) = K_1(\fk A_\Gamma)= \bb Z^{2r}\oplus T.
\end{equation}
\end{theorem}

\noindent Given $\xi\in P$, let
\begin{equation}\label{edgesum}
\ang {\xi} =
\sum_{\substack{a\in \widehat X^2 \\ a_2=\xi}}a
\in C(\Gamma).
\end{equation}
It is sometimes convenient to write such sums pictorially as
\begin{equation}\label{pict}
\ang {\xi}
=
\displaystyle
\sum_{
\beginpicture
\setcoordinatesystem units <.4cm, .7cm>
\setplotarea x from -1 to 1, y from 0 to 1 
\put {$_\bullet$} at 0 0
\put {$_a$}  at 0 0.6
\arrow <6pt> [.2, .67] from  0.3 0.3 to 0.5 0.5
\put {$_\xi$} at  0.9 0.4
\putrule from -1 1 to 1 1
\setlinear \plot  -1 1  0 0  1 1 /
\endpicture
}
a.
\end{equation}
Note that $a\in \widehat X^2$ with $a_2=\xi$ if and only if
$a=\ang{a_0,a_1,a_2}$ where $a_2=\xi$ and $a_0\in\ovl{\xi}$ (and $a_1$ is then uniquely determined).
There are $q+1$ such choices of $a_0$ and so there are $q+1$ terms in the sum (\ref{edgesum}). Similar remarks apply to the element
$\ang {\overline\xi}\in C(\Gamma)$ defined by
\begin{equation}\label{inverseedgesum}
\ang {\overline\xi} =
\sum_{\substack{a\in \widehat X^2 \\ a_1=\xi}}a
=
\sum_{
\beginpicture
\setcoordinatesystem units <.4cm, .7cm>
\setplotarea x from -1 to 1, y from 0 to 1 
\put {$_\bullet$} at 0 0
\put {$_a$}  at 0 0.6
\arrow <6pt> [.2, .67] from  -0.7 0.7 to -0.5 0.5
\put {$_\xi$} at  -0.8 0.4
\putrule from -1 1 to 1 1
\setlinear \plot  -1 1  0 0  1 1 /
\endpicture
}
a.
\end{equation}

\noindent In what follows the element
$$
\varepsilon =
\sum_{a\in \widehat X^2}a
$$
plays a special role.
An important observation, which is needed subsequently, is that $\varepsilon$ has finite order in $C(\Gamma)$.
The statement and its proof are very like \cite[Proposition 8.2]{rs3}.

\begin{lemma}\label{q2}
In the group $C(\Gamma)$,  $(q^2-1)\varepsilon = 0$.
\end{lemma}

\begin{proof}
Using relations (\ref{rels}) and the fact that each column of the matrix $M$ has
precisely $q^2$ nonzero entries,
  \begin{equation*}
  \e=\sum_{a\in \widehat X^2}a
=\sum_{a\in \widehat X^2}\sum_{b\in \widehat X^2}m_{ab}b= \sum_{b\in \widehat X^2}\left(\sum_{a\in \widehat X^2}m_{ab}\right)b =
\sum_{b\in \widehat X^2}q^2b
  = q^2\e.
\end{equation*}
\end{proof}
\begin{lemma}
If $\ang{a_0,a_1,a_2}\in \widehat X^2$ then, in the group $C(\Gamma)$,
\begin{subequations}\label{rel1}
\begin{eqnarray}
\ang {a_1} -\ang{a_2,a_0,a_1} &=& \ang{\overline a_2}- \ang{a_1,a_2,a_0} ; \label{rel1a}\\
\ang{a_0} + \ang{a_1} + \ang{a_2} &=& \varepsilon.
\label{rel1b}
\end{eqnarray}
\end{subequations}
\end{lemma}

\begin{proof}
Fix the base vertex $1\in \Delta$. Any generator $a$ for $C(\Gamma)$ has a unique representative
directed chamber $\sigma_a$ based at $1$.
The chamber $\sigma_a$ has vertices $1, a_1^{-1}, a_2$.
By \cite[Section 15.4]{gar}, each chamber based at $1$, other than $\sigma_a$ lies in a common apartment with $\sigma_a$, in exactly one of the five positions $\tau_2, \tau_3, \tau_4, \tau_5, \tau_6$ in Figure \ref{hexagon}.
As before, directed chambers will be pointed.

\refstepcounter{picture}
\begin{figure}[htbp]\label{hexagon}
\hfil
\centerline{
\beginpicture
\setcoordinatesystem units <1cm, 1.732cm>
\setplotarea  x from -2 to 2,  y from -1.5 to 1
\put {$_{a_1}$} [r ] at -0.6 0.5
\put {$_{a_2}$} [ l] at 0.6 0.5
\put {$_{a_0}$} [t] at  0.0   1.2
\arrow <10pt> [.2, .67] from  0.2 1 to 0 1
\arrow <10pt> [.2, .67] from  -0.7 0.7 to -0.5 0.5
\arrow <10pt> [.2, .67] from  0.3 0.3 to 0.5 0.5
\put {$\bullet$} at 0 0.3
\put {$1$} at 0.3 0.0
\put {$\tau_2$} at  1.1 0.3
\put {$\sigma_a$} at  0 0.6
\put {$\tau_3$} at  1.1 -0.3
\put {$\tau_4$} at  0 -0.7
\put {$\tau_5$} at -1.1 -0.3
\put {$\tau_6$} at -1.1 0.3
\putrule from -1 1 to 1 1
\putrule from -2 0 to 0.2 0  
\putrule from 0.4 0 to 2 0
\putrule from -1 -1 to 1 -1
\setlinear \plot 0 0  1 1  2 0  1 -1  0 0  -1 1  -2 0  -1 -1  0 0 /
\setshadegrid span <1.5pt>
\vshade   -1  1  1  <,z,,>  0  0  1  <z,,,>  1 1 1 /
\endpicture
}
\hfil
\caption{}
\end{figure}

\noindent The left side of (\ref{rel1a}) is equal to the sum of all the
elements $b\in \widehat X^2$ represented by directed chambers $\sigma_b$ in position $\tau_6$, as illustrated in Figure \ref{abchexagons}(a).
Each such element $b$ satisfies $b_2=a_1$, and the relations (\ref{rels}) imply that
$$
b= \sum_{c\in \widehat X^2}m_{bc}c.
$$
That is, $b$ is the sum of all the elements $c\in \widehat X^2$ with representative directed chambers $\sigma_c$ lying in position $\tau_4$, as illustrated in Figure \ref{abchexagons}(a).
Moreover, if $\sigma_c$ is any directed chamber with base vertex $1$, lying
in position $\tau_4$, then it arises in this way from a unique chamber $\sigma_b$ in position $\tau_6$. To see this, it is enough to take the convex hull of any such
chamber $\sigma_c$ with $\sigma_a$, which completely determines the whole hexagon in Figure \ref{abchexagons}(a).
Therefore the left side of (\ref{rel1a}) is equal to the sum of all the
elements $c\in \widehat X^2$ represented by directed chambers $\sigma_c$ based at $1$ which lie in position $\tau_4$ of Figure \ref{hexagon}.

\refstepcounter{picture}
\begin{figure}[htbp]\label{abchexagons}
\hfil
\centerline{
\beginpicture
\setcoordinatesystem units <0.7cm, 1.2cm>
\setplotarea  x from -2 to 2,  y from -1.5 to 1
\put {(a)} at 0 -1.4
\put {$\bullet$} at -1 0.7
\put {$\bullet$} at 0 -0.3
\arrow <7pt> [.2, .67] from  -0.7 0.7 to -0.5 0.5
\put {$\sigma_c$} at  0 -0.7
\put {$\sigma_b$} at -1.1 0.3
\putrule from -1 1 to 1 1
\putrule from -2 0 to 2 0
\putrule from -1 -1 to 1 -1
\setlinear \plot 0 0  1 1  2 0  1 -1  0 0  -1 1  -2 0  -1 -1  0 0 /
\setshadegrid span <1.5pt>
\vshade   -1  1  1  <,z,,>  0  0  1  <z,,,>  1 1 1 /
\endpicture
\qquad\qquad
\beginpicture
\setcoordinatesystem units <0.7cm, 1.2cm>
\setplotarea  x from -2 to 2,  y from -1.5 to 1
\put {(b)} at 0 -1.4
\put {$\bullet$} at 0 -0.3
\put {$\bullet$} at 1 0.7
\arrow <7pt> [.2, .67] from  0.4 0.4 to 0.6 0.6
\put {$\sigma_b$} at  1.1 0.3
\put {$\sigma_c$} at  0 -0.7
\putrule from -1 1 to 1 1
\putrule from -2 0 to 2 0
\putrule from -1 -1 to 1 -1
\setlinear \plot 0 0  1 1  2 0  1 -1  0 0  -1 1  -2 0  -1 -1  0 0 /
\setshadegrid span <1.5pt>
\vshade   -1  1  1  <,z,,>  0  0  1  <z,,,>  1 1 1 /
\endpicture
}
\hfil
\caption{}
\end{figure}

\noindent Similarly, the right side of (\ref{rel1a}) is equal to the sum of all the  elements $b\in \widehat X^2$ represented by directed chambers $\sigma_b$ in position $\tau_2$ as illustrated in Figure \ref{abchexagons}(b).
The relations (\ref{rels}) imply that, for each such chamber $b$,
$$
b= \sum_{b\in \widehat X^2}n_{ac}c.
$$
It follows that the right side of (\ref{rel1a}) is also equal to the sum of all the elements $c\in \widehat X^2$ represented by directed chambers based at $1$ which lie in position $\tau_4$ of Figure \ref{hexagon}.
This proves that the left and right sides of (\ref{rel1a}) are equal.

The next task is to prove (\ref{rel1b}).
Recall that $\varepsilon$ is the sum of all the elements of $\widehat X^2$,
and representative directed chambers for elements of this sum are $\sigma_a$  together with all chambers lying in any of the five positions $\tau_2, \tau_3, \tau_4, \tau_5, \tau_6$ in Figure \ref{hexagon}.

The set of chambers based at the vertex $1$ representing the elements of the sum $\ang{a_2}$
consists of $\sigma_a$ together with all directed chambers lying in the position $\tau_2$, as illustrated in Figure \ref{subhexagons}(a).
Here it may also be convenient to refer back to equation (\ref{pict}).

\refstepcounter{picture}
\begin{figure}[htbp]\label{subhexagons}
\hfil
\centerline{
\beginpicture
\setcoordinatesystem units <0.7cm, 1.2cm>
\setplotarea  x from -2 to 2,  y from -1.5 to 1
\put {(a)} at 0 -1.4
\put {$\bullet$} at 0 0.3
\put {$\bullet$} at 0.5 0.15
\put {$_{a_2}$} at 0.4 0.7
\arrow <7pt> [.2, .67] from  0.4 0.4 to 0.6 0.6
\put {$\tau_2$} at  1.1 0.3
\put {$\tau_3$} at  1.1 -0.3
\put {$\tau_4$} at  0 -0.7
\put {$\tau_5$} at -1.1 -0.3
\put {$\tau_6$} at -1.1 0.3
\putrule from -1 1 to 1 1
\putrule from -2 0 to 2 0
\putrule from -1 -1 to 1 -1
\setlinear \plot 0 0  1 1  2 0  1 -1  0 0  -1 1  -2 0  -1 -1  0 0 /
\setshadegrid span <1.5pt>
\vshade   -1  1  1  <,z,,>  0  0  1  <z,,,>  1 1 1 /
\endpicture
\qquad
\beginpicture
\setcoordinatesystem units <0.7cm, 1.2cm>
\setplotarea  x from -2 to 2,  y from -1.5 to 1
\put {(b)} at 0 -1.4
\put {$\bullet$} at 0 -0.3
\put {$\bullet$} at 0.5 -0.15
\put {$_{a_1}$} at -0.2 0.5
\arrow <7pt> [.2, .67] from  -0.7 0.7 to -0.5 0.5
\put {$\tau_2$} at  1.1 0.3
\put {$\tau_3$} at  1.1 -0.3
\put {$\tau_4$} at  0 -0.7
\put {$\tau_5$} at -1.1 -0.3
\put {$\tau_6$} at -1.1 0.3
\putrule from -1 1 to 1 1
\putrule from -2 0 to 2 0
\putrule from -1 -1 to 1 -1
\setlinear \plot 0 0  1 1  2 0  1 -1  0 0  -1 1  -2 0  -1 -1  0 0 /
\setshadegrid span <1.5pt>
\vshade   -1  1  1  <,z,,>  0  0  1  <z,,,>  1 1 1 /
\endpicture
\qquad
\beginpicture
\setcoordinatesystem units <0.7cm, 1.2cm>
\setplotarea  x from -2 to 2,  y from -1.5 to 1
\put {(c)} at 0 -1.4
\put {$\bullet$} at -0.5 0.15
\put {$\bullet$} at -0.5 -0.15
\put {$_{a_0}$} at  0.0   0.83
\arrow <7pt> [.2, .67] from  0.2 1 to 0 1
\put {$\tau_2$} at  1.1 0.3
\put {$\tau_3$} at  1.1 -0.3
\put {$\tau_4$} at  0 -0.7
\put {$\tau_5$} at -1.1 -0.3
\put {$\tau_6$} at -1.1 0.3
\putrule from -1 1 to 1 1
\putrule from -2 0 to 2 0
\putrule from -1 -1 to 1 -1
\setlinear \plot 0 0  1 1  2 0  1 -1  0 0  -1 1  -2 0  -1 -1  0 0 /
\setshadegrid span <1.5pt>
\vshade   -1  1  1  <,z,,>  0  0  1  <z,,,>  1 1 1 /
\endpicture
}
\hfil
\caption{}
\end{figure}
\noindent Using the relations (\ref{rels}),
the sum $\ang{a_1}$ is equal to the sum of elements represented by
chambers lying in the position $\tau_3$ or $\tau_4$,
as in Figure \ref{subhexagons}(b).

Finally, the sum $\ang{a_0}$ is equal to the sum of elements
represented by directed chambers lying in the position $\tau_5$ or $\tau_6$,
as in Figure \ref{subhexagons}(c).
For, cyclically permuting the indices in the equation (\ref{rel1a}) gives
$$
\ang {a_0} -\ang{a_1,a_2,a_0} = \ang{\overline a_1}- \ang{a_0,a_1,a_2}.
$$
Therefore $\ang {a_0} =\ang{a_1,a_2,a_0} + \ang{\overline a_1}- \ang{a_0,a_1,a_2}$.
Referring to  Figure \ref{subhexagons}(c), the relations (\ref{rels}) show that $\ang{a_1,a_2,a_0}$ is the sum of elements
represented by directed chambers in position $\tau_5$.
Also $\ang{\overline a_1}- \ang{a_0,a_1,a_2}$ is
the sum of elements represented by directed chambers in position $\tau_6$.

This completes the proof that $\ang{a_0} + \ang{a_1} + \ang{a_2} = \varepsilon$.
\end{proof}

\noindent The next lemma is a major step in the proof of the main theorem.
It depends on the fact that $\Gamma$ has Kazhdan's property (T),
which in turn depends only on the local structure of the building $\Delta$.
See, for example, the proof of \cite[Theorem 5.7.7]{bhv}.

\begin{lemma}\label{33}
In the group $C(\Gamma)\otimes\bb R$,
for all $\ang{a_0,a_1,a_2}\in \widehat X^2$ and $\xi \in P$,
  \begin{subequations}\label{relz1}
\begin{eqnarray}
&&\ang{a_0,a_1,a_2}\otimes 1=\ang{a_1,a_2,a_0}\otimes 1=\ang{a_2,a_0,a_1}\otimes 1; \label{relza}\\
&&\ang{\xi}\otimes 1 = 0 = \ang{\ovl\xi}\otimes 1. \label{relzb}
\end{eqnarray}
\end{subequations}
\end{lemma}

\begin{proof}
By Lemma \ref{q2}, $\varepsilon$ has finite order in $C(\Gamma)$
and hence $\varepsilon\otimes 1$ is zero in $C(\Gamma)\otimes\bb R$.
Therefore, by (\ref{rel1b}),
$$
\ang{a_0}\otimes 1 + \ang{a_1}\otimes 1 + \ang{a_2}\otimes 1 = 0,
$$
for $\ang{a_0,a_1,a_2}\in \widehat X^2$. It follows from the presentation of $\Gamma$ that the map $\xi\mapsto \ang{\xi}\otimes 1$, $\xi \in P$, induces a
homomorphism $\theta$ from $\Gamma$ into the abelian group $C(\Gamma)\otimes\bb R$.

The $\widetilde A_2$ group $\Gamma$ has Kazhdan's property (T), by \cite[Theorem 5.7.7]{bhv}. It follows that the range of $\theta$ is finite \cite[Corollary 1.3.6]{bhv} and hence zero, since
$C(\Gamma)\otimes\bb R$ is torsion free.
Therefore $\ang{\xi}\otimes 1=0$, $\xi \in P$.
Similarly, $\ang{\overline \xi}\otimes 1=0$, $\xi \in P$.
This proves (\ref{relzb}).
The relation (\ref{rel1a}) then implies that
$\ang{a_1,a_2,a_0}\otimes 1=\ang{a_2,a_0,a_1}\otimes 1$ and the rest of (\ref{relza}) follows by symmetry.
\end{proof}

Let $C_0(\Gamma)$ be the abelian group with generating set $\widehat X^2$ and the following relations:
\begin{subequations}\label{rel0}
\begin{eqnarray}
&&\ang{a_0,a_1,a_2} = \ang{a_1,a_2,a_0}=\ang{a_2,a_0,a_1}, \qquad \ang{a_0,a_1,a_2}\in\widehat X^2 ; \label{rel0a}\\
&&\ang{\xi} = 0 = \ang{\ovl\xi}, \qquad \xi \in P. \label{rel0b}
\end{eqnarray}
\end{subequations}

\begin{lemma}\label{presentation0}
The relations (\ref{rel0}) imply the relations (\ref{rels}).
\end{lemma}

\begin{proof}
  Let $a=\ang{a_0,a_1,a_2}\in \widehat X^2$. Then, using the relations (\ref{rel0}), and referring to Figure \ref{xxx},
\begin{eqnarray*}
a
&=& \ang{a_0,a_1,a_2} = \ang{a_1,a_2,a_0}  \hbox {\qquad [using (\ref{rel0a})] } \\
&=& -\sum_{
\substack{\ang{c_2,b_1,a_0}\in\widehat X^2 \\ c_2\ne a_1}
}
\ang{c_2,b_1,a_0}
\hbox {\qquad [using (\ref{rel0b}), with $\xi=a_0$] }
\\
&=& -\sum_{
\substack{\ang{c_2,b_1,a_0}\in\widehat X^2 \\ c_2\ne a_1}
}
\left(
-\sum_{
\substack{\ang{b_0,b_1,b_2}\in\widehat X^2 \\ b_0\ne c_2}
}
\ang{b_0,b_1,b_2}  \right)
\hbox {\quad [using (\ref{rel0b}) again] } \\
&=& \sum_{b\in \widehat X^2}m_{ab}b.
\end{eqnarray*}

\refstepcounter{picture}
\begin{figure}[htbp]\label{xxx}
\hfil
\centerline{
\beginpicture
\setcoordinatesystem units <0.8cm, 1.4cm>
\setplotarea  x from -2 to 2,  y from -0.8 to 1
\put {$\bullet$} at 0.65 -0.15
\put {$\bullet$} at 0.7 0.1
\put {$\bullet$} at 1 0.25
\put {$_{a_1}$} [] at -0.8 -0.5
\put {$_{a_2}$} [] at 0.8 -0.5
\put {$_{b_1}$} [] at 0.35  0.5
\put {$_{b_2}$} [] at 1.8  0.5
\put {$_{a_0}$} [] at  0.0   -0.15
\put {$_{b_0}$} [] at  1.0   0.9
\put {$_{c_2}$} [] at -0.25  0.5
\arrow <10pt> [.2, .67] from  -0.7 0.3 to -0.5 0.5
\multiput {\beginpicture
\setcoordinatesystem units <0.8cm, 1.4cm>
\arrow <10pt> [.2, .67] from  0.2 1 to 0 1
\arrow <10pt> [.2, .67] from  -0.7 0.7 to -0.5 0.5
\arrow <10pt> [.2, .67] from  0.3 0.3 to 0.5 0.5
\putrule from -1 1 to 1 1
\setlinear \plot -1 1 0 0 1 1 /
\endpicture}  at   0 0  1 1  /
\plot -1 0  0 1 /
\endpicture
}
\hfil
\caption{}
\end{figure}

\noindent The proof of the relations
$a=\sum_{b\in \widehat X^2}n_{ab}b$ in (\ref{rels}) is similar.
\end{proof}

\begin{proposition}\label{C=C0}
If $\Gamma$ is a torsion free $\widetilde A_2$ group, then
  $C(\Gamma)\otimes\bb R = C_0(\Gamma)\otimes\bb R$.
\end{proposition}

\begin{proof}
The groups have the same set of generators.
By Lemmas \ref{33} and  \ref{presentation0}, the relations in each group
imply the relations in the other. The groups are therefore equal.
\end{proof}

\section{Harmonic cochains and proof of the main result}

A harmonic 2-cochain \cite{eck} is a function
$c: \widehat X^2 \to \bb R$ satisfying the following conditions
for all $a\in\widehat X^2$ and for all $\xi\in P$.
\begin{subequations}\label{harmonic}
\begin{eqnarray}
c(\ang{a_0,a_1,a_2})&=&c(\ang{a_1,a_2,a_0})\,\,=\,\,c(\ang{a_2,a_0,a_1}); \label{harmonic(a)}\\
c(\ang {\xi})&=&c(\ang {\ovl\xi})\,\,=\,\,0.
\label{harmonic(b)}
\end{eqnarray}
\end{subequations}

Denote the set of harmonic 2-cochains by $C^2_{har}(\widehat X^2)$.
Since the group $\Gamma$ acts freely on $\Delta$, $C^2_{har}(\widehat X^2)$ may be identified with the space of $\Gamma$-invariant harmonic cochains $c:\widehat\Delta^2\to \bb R$, in the sense of \cite{ads}.
Now $C^2_{har}(\widehat X^2)$ is the algebraic dual of $C_0(\Gamma)\otimes\bb R$.
The next result is therefore an immediate consequence of
Proposition \ref{C=C0}.
\begin{proposition}\label{harm}
$C^2_{har}(\widehat X^2)$ is isomorphic to $C(\Gamma)\otimes\bb R$.
\end{proposition}

The proof of Theorem \ref{main} can now be completed.
By Theorem \ref{Ki}, it is sufficient to show that the rank $r$ of
$C(\Gamma)$ is equal to $\beta_2=\dim_\bb R H^2(\Gamma, \bb R)$.
Garland's isomorphism \cite[Section 3.1]{ads} states that $H^2(\Gamma,\bb R)\cong C^2_{har}(\widehat X^2)$. Note that the account of Garland's Theorem in
\cite{ads} relates to the case where $\Gamma$ is a lattice in $\PGL (3,\bb K)$,
but the proof applies without change to all torsion free $\widetilde A_2$ groups.

 It follows from Proposition \ref{harm} that
$C(\Gamma)\otimes\bb R\cong H^2(\Gamma,\bb R)$.
Theorem \ref{Ki} now implies that
$K_0(\fk A_\Gamma)\otimes \bb R \cong \bb R^{2\beta_2}$.

It remains to identify $\beta_2$ explicitly.
The Euler characteristic of $\Gamma$ is $\chi(\Gamma)= \frac{1}{3}(q-1)(q^2-1)$ \cite[Section 4]{ro2001}.
Now $\chi(\Gamma)=\beta_0-\beta_1+\beta_2$ where $\beta_i=\dim_\bb R H_i(\Gamma, \bb R)$.
Since $\Gamma$ has Kazhdan's property (T),
the abelianisation $\Gamma/[\Gamma,\Gamma]$ is finite \cite[Corollary 1.3.6]{bhv},
and so $\beta_1=0$. Also $\beta_0=1$. Therefore
$\beta_2=\chi(\Gamma)-1=\frac{1}{3}(q-2)(q^2+q+1)$.
This completes the proof.
\qed

\end{document}